\documentclass{article}
\usepackage{amssymb}
\usepackage{amsfonts}
\usepackage{amsmath}

\newtheorem{theorem}{Theorem}

\newtheorem{corollary}[theorem]{Corollary}

\newtheorem{definition}[theorem]{Definition}

\newtheorem{lemma}[theorem]{Lemma}

\newenvironment{proof}[1][Proof]{\noindent\textbf{#1.} }{\ \rule{0.5em}{0.5em}}
\input{tcilatex}
\begin{document}

\title{On some subclasses of m-fold symmetric bi-univalent functions}
\author{\c{S}ahsene Alt\i nkaya$^{\ast }$ and Sibel Yal\c{c}\i n$^{1}$ \\
%EndAName
$^{\ast ,1}$Department of Mathematics, Faculty of Arts and Science,\\
Uludag University, Bursa, Turkey.\\
$^{\ast }$corresponding author e-mail: sahsene@uludag.edu.tr\\
$^{1}$syalcin@uludag.edu.tr}
\maketitle

\begin{abstract}
In this work, we introduce two new subclasses $S_{\Sigma _{m}}(\alpha
,\lambda )$ and $S_{\Sigma _{m}}(\beta ,\lambda )$ of $\Sigma _{m}$
consisting of analytic and \textit{m}-fold symmetric bi-univalent functions
in the open unit disk $U$. Furthermore, for functions in each of the
subclasses introduced in this paper, we obtain the coefficient bounds for $%
\left\vert a_{m+1}\right\vert $ and $\left\vert a_{2m+1}\right\vert .$

Keywords: Analytic functions, \textit{m}-fold symmetric bi-univalent
functions, Coefficient bounds.

2010, Mathematics Subject Classification: 30C45, 30C50.
\end{abstract}

\section{Introduction}

Let $A$ denote the class of functions of the form 
\begin{equation}
f(z)=z+\overset{\infty }{\underset{n=2}{\sum }}a_{n}z^{n}  \label{eq1}
\end{equation}%
which are analytic in the open unit disk $U=\left\{ z:\left\vert
z\right\vert <1\right\} $, and let $S$ be the subclass of $A$ consisting of
the form (\ref{eq1}) which are also univalent in $U.$

The Koebe one-quarter theorem \cite{Duren 83} states that the image of $U$
under every function $f$ \ from $S$ contains a disk of radius $\frac{1}{4}.$
Thus every such univalent function has an inverse $f^{-1}$ which satisfies%
\begin{equation*}
f^{-1}\left( f\left( z\right) \right) =z~~\left( z\in U\right)
\end{equation*}%
and%
\begin{equation*}
f\left( f^{-1}\left( w\right) \right) =w~~\left( \left\vert w\right\vert
<r_{0}\left( f\right) ~,~r_{0}\left( f\right) \geq \frac{1}{4}\right) ,
\end{equation*}%
where%
\begin{equation}
f^{-1}\left( w\right) =w~-a_{2}w^{2}+\left( 2a_{2}^{2}-a_{3}\right)
w^{3}-\left( 5a_{2}^{3}-5a_{2}a_{3}+a_{4}\right) w^{4}+\cdots .  \label{eq2}
\end{equation}

A function $f\in A$ is said to be bi-univalent in $U$ if both $f$ and $%
f^{-1} $ are univalent in $U.~$Let $\Sigma $ denote the class of
bi-univalent functions defined in the unit disk $U.$

For a brief history and interesting examples in the class $\Sigma ,$ see 
\cite{Srivastava 2010}. Examples of functions in the class $\Sigma $ are%
\begin{equation*}
\frac{z}{1-z},\ \ -\log (1-z),\ \ \ \frac{1}{2}\log \left( \frac{1+z}{1-z}%
\right)
\end{equation*}%
and so on. However, the familier Koebe function is not a member of $\Sigma .$
Other common examples of functions in $S$ such as%
\begin{equation*}
z-\frac{z^{2}}{2}\text{ and }\frac{z}{1-z^{2}}
\end{equation*}%
are also not members of $\Sigma $ (see \cite{Srivastava 2010}).

For each function $f\in S$, the function%
\begin{equation}
h(z)=\sqrt[m]{f(z^{m})}\ \ \ \ \ \ \ \ \ \ \ \ \ (z\in U,\ \ m\in 
%TCIMACRO{\U{2115} }%
%BeginExpansion
\mathbb{N}
%EndExpansion
)  \label{eq3}
\end{equation}%
is univalent and maps the unit disk $U$ into a region with \textit{m}-fold
symmetry. A function is said to be \textit{m}-fold symmetric (see \cite%
{Koepf 89}, \cite{Pommerenke 75}) if it has the following normalized form:%
\begin{equation}
f(z)=z+\overset{\infty }{\underset{k=1}{\sum }}a_{mk+1}z^{mk+1}\ \ \ \ \ \ \
(z\in U,\ \ m\in 
%TCIMACRO{\U{2115} }%
%BeginExpansion
\mathbb{N}
%EndExpansion
).  \label{eq4}
\end{equation}

We denote by $S_{m}$ the class of \textit{m}-fold symmetric univalent
functions in U, which are normalized by the series expansion (\ref{eq4}). In
fact, the functions in the class $S$ are \textit{one}-fold symmetric.

Analogous to the concept of \textit{m}-fold symmetric univalent functions,
we here introduced the concept of \textit{m}-fold symmetric bi-univalent
functions. Each function $f\in \Sigma $ generates an \textit{m}-fold
symmetric bi-univalent function for each integer $m\in 
%TCIMACRO{\U{2115} }%
%BeginExpansion
\mathbb{N}
%EndExpansion
$. The normalized form of f is given as in (\ref{eq4}) and the series
expansion for $f^{-1},$ which has been recently proven by Srivastava et al. 
\cite{Srivastava 2014}, is given as follows:%
\begin{eqnarray}
g\left( w\right) &=&w~-a_{m+1}w^{m+1}+\left[ (m+1)a_{m+1}^{2}-a_{2m+1}\right]
w^{2m+1}  \label{eq5} \\
&&-\left[ \frac{1}{2}(m+1)(3m+2)a_{m+1}^{3}-(3m+2)a_{m+1}a_{2m+1}+a_{3m+1}%
\right] w^{3m+1}+\cdots .  \notag
\end{eqnarray}%
where $f^{-1}=g.$ We denote by $\Sigma _{m}$ the class of \textit{m}-fold
symmetric bi-univalent functions in $U$. For $m=1$, the formula (\ref{eq5})
coincides with the formula (\ref{eq2}) of the class $\Sigma $. Some examples
of \textit{m}-fold symmetric bi-univalent functions are given as follows:%
\begin{equation*}
\left( \frac{z^{m}}{1-z^{m}}\right) ^{m},\ \ \left[ -\log (1-z^{m})\right] ^{%
\frac{1}{m}},\ \ \ \left[ \frac{1}{2}\log \left( \frac{1+z^{m}}{1-z^{m}}%
\right) ^{\frac{1}{m}}\right] .
\end{equation*}

Brannan and Taha \cite{Brannan and Taha 86} introduced certain subclasses of
the bi-univalent function class $\Sigma $ similar to the familiar
subclasses. $S^{\star }\left( \beta \right) $ and $K\left( \beta \right) $
of starlike and convex function of order $\beta $ $\left( 0\leq \beta
<1\right) $ respectively (see \cite{Netanyahu 69}). The classes $S_{\Sigma
}^{\star }\left( \alpha \right) $ and $K_{\Sigma }\left( \alpha \right) $ of
bi-starlike functions of order $\alpha $ and bi-convex functions of order $%
\alpha ,$ corresponding to the function classes $S^{\star }\left( \alpha
\right) $ and $K\left( \alpha \right) ,$ were also introduced analogously.
For each of the function classes $S_{\Sigma }^{\star }\left( \alpha \right) $
and $K_{\Sigma }\left( \alpha \right) ,$ they found non-sharp estimates on
the initial coefficients. In fact, the aforecited work of Srivastava et al. 
\cite{Srivastava 2010} essentially revived the investigation of various
subclasses of the bi-univalent function class $\Sigma $ in recent years.
Recently, many authors investigated bounds for various subclasses of
bi-univalent functions (\cite{Altinkaya and Yalcin 2014}, \cite{Altinkaya
and Yalcin 2015}, \cite{Frasin 2011}, \cite{Magesh and Yamini 2013}, \cite%
{Srivastava 2010}, \cite{Srivastava 2013},). Not much is known about the
bounds on the general coefficient$\ \left\vert a_{n}\right\vert $ for $n\geq
4.$ In the literature, there are only a few works determining the general
coefficient bounds $\left\vert a_{n}\right\vert $ for the analytic
bi-univalent functions (\cite{Altinkaya and Yalcin 2015a}, \cite{Bulut 2014}%
, \cite{Hamidi and Jahangiri 2014}). The coefficient estimate problem for
each of $\left\vert a_{n}\right\vert $ $\left( \ n\in 
%TCIMACRO{\U{2115} }%
%BeginExpansion
\mathbb{N}
%EndExpansion
\backslash \left\{ 1,2\right\} ;\ \ 
%TCIMACRO{\U{2115} }%
%BeginExpansion
\mathbb{N}
%EndExpansion
=\left\{ 1,2,3,...\right\} \right) $ is still an open problem.

The aim of the this paper is to introduce two new subclasses of the function
class $\Sigma _{m}$ and derive estimates on the initial coefficients $%
\left\vert a_{m+1}\right\vert $ and $\left\vert a_{2m+1}\right\vert $ for
functions in these new subclasses. We have to remember here the following
lemma here so as to derive our basic results:

\begin{lemma}
\cite{Pommerenke 75} If $p\left( z\right)
=1+p_{1}z+p_{2}z^{2}+p_{3}z^{3}+\cdots \ \ $is an analytic function in $U$
with positive real part, then%
\begin{equation*}
\left\vert p_{n}\right\vert \leq 2~\ \ \ \ \ \left( n\in 
%TCIMACRO{\U{2115} }%
%BeginExpansion
\mathbb{N}
%EndExpansion
=\left\{ 1,2,\ldots \right\} \right)
\end{equation*}%
and%
\begin{equation*}
\left\vert p_{2}-\frac{p_{1}^{2}}{2}\right\vert \leq 2-\frac{\left\vert
p_{1}\right\vert ^{2}}{2}.\ 
\end{equation*}
\end{lemma}

\section{Coefficient bounds for the function class $S_{\Sigma _{m}}(\protect%
\alpha ,\protect\lambda )$}

\begin{definition}
A function $f\in \Sigma _{m}$ is said to be in the class $S_{\Sigma
_{m}}(\alpha ,\lambda )$ if the following conditions are satisfied:%
\begin{equation*}
\begin{array}{cc}
\left\vert \arg \left[ \frac{1}{2}\left( \frac{zf^{\prime }(z)}{f(z)}+\left( 
\frac{zf^{\prime }(z)}{f(z)}\right) ^{\frac{1}{\lambda }}\right) \right]
\right\vert <\dfrac{\alpha \pi }{2} & \left( 0<\alpha \leq 1,\ 0<\lambda
\leq 1,\ \ z\in U\right)%
\end{array}%
\ 
\end{equation*}%
and%
\begin{equation*}
\begin{array}{cc}
\left\vert \arg \left[ \frac{1}{2}\left( \frac{wg^{\prime }(w)}{g(w)}+\left( 
\frac{wg^{\prime }(w)}{g(w)}\right) ^{\frac{1}{\lambda }}\right) \right]
\right\vert <\dfrac{\alpha \pi }{2} & \left( 0<\alpha \leq 1,\ 0<\lambda
\leq 1,\ w\in U\right)%
\end{array}%
\end{equation*}%
where the function $g=f^{-1}.$
\end{definition}

\begin{theorem}
Let $\ f$ given by (\ref{eq4}) be in the class $S_{\Sigma _{m}}(\alpha
,\lambda ),\ 0<\alpha \leq 1.$ Then 
\begin{equation*}
\left\vert a_{m+1}\right\vert \leq \dfrac{4\lambda \alpha }{m\sqrt{%
(1+\lambda )\left[ 4\lambda \alpha +(1+\lambda )(1-\alpha )\right] +2\alpha
(1-\lambda )}}
\end{equation*}%
and%
\begin{equation*}
\left\vert a_{2m+1}\right\vert \leq \frac{2\lambda \alpha }{m\left(
1+\lambda \right) }+\dfrac{8(m+1)\lambda ^{2}\alpha ^{2}}{m^{2}(1+\lambda
)^{2}}.
\end{equation*}
\end{theorem}

\begin{proof}
Let $\ f\in S_{\Sigma _{m}}(\alpha ,\lambda ).$ Then 
\begin{equation}
\frac{1}{2}\left( \frac{zf^{\prime }(z)}{f(z)}+\left( \frac{zf^{\prime }(z)}{%
f(z)}\right) ^{\frac{1}{\lambda }}\right) =\left[ p(z)\right] ^{\alpha }
\label{eq6}
\end{equation}%
\begin{equation}
\frac{1}{2}\left( \frac{wg^{\prime }(w)}{g(w)}+\left( \frac{wg^{\prime }(w)}{%
g(w)}\right) ^{\frac{1}{\lambda }}\right) =\left[ q(w)\right] ^{\alpha }
\label{eq7}
\end{equation}%
where $g=f$ $^{-1}$, $p,q$ in $P$ and have the forms 
\begin{equation*}
p(z)=1+p_{m}z^{m}+p_{2m}z^{2m}+\cdots
\end{equation*}%
and%
\begin{equation*}
q(w)=1+q_{m}w^{m}+q_{2m}w^{2m}+\cdots .
\end{equation*}%
Now, equating the coefficients in (\ref{eq6}) and (\ref{eq7}), we get%
\begin{equation}
\frac{m(1+\lambda )}{2\lambda }a_{m+1}=\alpha p_{m},  \label{eq8}
\end{equation}%
\begin{equation}
\frac{m(1+\lambda )}{2\lambda }\left( 2a_{2m+1}-a_{m+1}^{2}\right) +\frac{%
m^{2}(1-\lambda )}{4\lambda ^{2}}a_{m+1}^{2}=\alpha p_{2m}+\tfrac{\alpha
(\alpha -1)}{2}p_{m}^{2},  \label{eq9}
\end{equation}%
and%
\begin{equation}
-\frac{m(1+\lambda )}{2\lambda }a_{m+1}=\alpha q_{m},  \label{eq10}
\end{equation}%
\begin{equation}
\frac{m(1+\lambda )}{2\lambda }\left[ (2m+1)a_{m+1}^{2}-2a_{2m+1}\right] +%
\frac{m^{2}(1-\lambda )}{4\lambda ^{2}}a_{m+1}^{2}=\alpha q_{2m}+\tfrac{%
\alpha (\alpha -1)}{2}q_{m}^{2}.  \label{eq11}
\end{equation}%
From (\ref{eq8}) and (\ref{eq10}) we obtain%
\begin{equation}
p_{m}=-q_{m}.  \label{eq12}
\end{equation}%
and%
\begin{equation}
\frac{m^{2}(1+\lambda )^{2}}{2\lambda ^{2}}a_{m+1}^{2}=\alpha
^{2}(p_{m}^{2}+q_{m}^{2}).  \label{eq13}
\end{equation}%
Also form (\ref{eq9}), (\ref{eq11}) and (\ref{eq13}) we have 
\begin{equation*}
\begin{array}{l}
\left[ \frac{m^{2}(1+\lambda )}{\lambda }+\frac{m^{2}\left( 1-\lambda
\right) }{2\lambda ^{2}}\right] a_{m+1}^{2}=\alpha \left(
p_{2m}+q_{2m}\right) +\frac{\alpha (\alpha -1)}{2}(p_{m}^{2}+q_{m}^{2}). \\ 
\\ 
\ \ \ \ \ \ \ \ \ \ \ \ \ \ \ \ \ \ \ \ \ \ \ \ \ \ \ \ \ \ \ \ \ \ \
=\alpha \left( p_{2m}+q_{2m}\right) +\frac{\alpha (\alpha -1)}{2}\frac{%
m^{2}(1+\lambda )^{2}}{2\lambda ^{2}\alpha ^{2}}a_{m+1}^{2}.%
\end{array}%
\end{equation*}%
Therefore, we have%
\begin{equation}
a_{m+1}^{2}=\dfrac{4\lambda ^{2}\alpha ^{2}\left( p_{2m}+q_{2m}\right) }{%
m^{2}\left\{ (1+\lambda )\left[ 4\lambda \alpha +(1+\lambda )(1-\alpha )%
\right] +2\alpha (1-\lambda )\right\} }.  \label{eq14}
\end{equation}%
Appying Lemma 1 for the coefficients $p_{2m}$ and $q_{2m}$, we obtain%
\begin{equation*}
\left\vert a_{m+1}\right\vert \leq \dfrac{4\lambda \alpha }{m\sqrt{%
(1+\lambda )\left[ 4\lambda \alpha +(1+\lambda )(1-\alpha )\right] +2\alpha
(1-\lambda )}}.
\end{equation*}%
Next, in order to find the bound on $\left\vert a_{2m+1}\right\vert ,$ by
subtracting (\ref{eq11}) from (\ref{eq9}), we obtain%
\begin{equation*}
\dfrac{2m(1+\lambda )}{\lambda }a_{2m+1}-\dfrac{m(m+1)(1+\lambda )}{\lambda }%
a_{m+1}^{2}=\alpha \left( p_{2m}-q_{2m}\right) +\tfrac{\alpha (\alpha -1)}{2}%
(p_{m}^{2}-q_{m}^{2}).
\end{equation*}%
Then, in view of (\ref{eq12}) and (\ref{eq13}) , and appying Lemma 1 for the
coefficients $p_{2m},p_{m}$ and $q_{2m},q_{m\text{ ,}}$we have%
\begin{equation*}
\left\vert a_{2m+1}\right\vert \leq \frac{2\lambda \alpha }{m\left(
1+\lambda \right) }+\dfrac{8(m+1)\lambda ^{2}\alpha ^{2}}{m^{2}(1+\lambda
)^{2}}.
\end{equation*}%
This completes the proof of Theorem 3.
\end{proof}

\section{Coefficient bounds for the function class $S_{\Sigma _{m}}(\protect%
\beta ,\protect\lambda )$}

\begin{definition}
A function $f\in \Sigma _{m}$ given by (\ref{eq4}) is said to be in the
class $S_{\Sigma _{m}}(\beta ,\lambda )$ if the following conditions are
satisfied:%
\begin{equation}
\func{Re}\left\{ \frac{1}{2}\left( \frac{zf^{\prime }(z)}{f(z)}+\left( \frac{%
zf^{\prime }(z)}{f(z)}\right) ^{\frac{1}{\lambda }}\right) \right\} >\beta
,\ \ \ \ \ \left( 0\leq \beta <1,\ \ 0<\lambda \leq 1,\ \ z\in U\right) \ 
\label{eq15}
\end{equation}%
and%
\begin{equation}
\func{Re}\left\{ \frac{1}{2}\left( \frac{wg^{\prime }(w)}{g(w)}+\left( \frac{%
wg^{\prime }(w)}{g(w)}\right) ^{\frac{1}{\lambda }}\right) \right\} >\beta
,\ \ \ \ \ \ \left( 0\leq \beta <1,\ \ 0<\lambda \leq 1,\ \ w\in U\right) .
\label{eq16}
\end{equation}%
where the function $g=f^{-1}.$
\end{definition}

\begin{theorem}
Let $\ f$ given by (\ref{eq4}) be in the class $S_{\Sigma _{m}}(\beta
,\lambda ),\ 0\leq \beta <1$. Then 
\begin{equation*}
\left\vert a_{m+1}\right\vert \leq \dfrac{2\lambda }{m}\sqrt{\dfrac{2\left(
1-\beta \right) }{2\lambda ^{2}+\lambda +1}}
\end{equation*}%
and%
\begin{equation*}
\left\vert a_{2m+1}\right\vert \leq \dfrac{8(m+1)\lambda ^{2}\left( 1-\beta
\right) ^{2}}{m^{2}(1+\lambda )^{2}}+\frac{2\lambda \left( 1-\beta \right) }{%
m\left( 1+\lambda \right) }.
\end{equation*}
\end{theorem}

\begin{proof}
Let $\ f\in S_{\Sigma _{m}}(\beta ,\lambda ).$ Then%
\begin{equation}
\frac{1}{2}\left( \frac{zf^{\prime }(z)}{f(z)}+\left( \frac{zf^{\prime }(z)}{%
f(z)}\right) ^{\frac{1}{\lambda }}\right) =\beta +(1-\beta )p(z)
\label{eq17}
\end{equation}%
\begin{equation}
\frac{1}{2}\left( \frac{wg^{\prime }(w)}{g(w)}+\left( \frac{wg^{\prime }(w)}{%
g(w)}\right) ^{\frac{1}{\lambda }}\right) =\beta +(1-\beta )q(w)
\label{eq18}
\end{equation}%
where $p,q\in P$ and $g=f$ $^{-1}.$

It follows from (\ref{eq17}) and (\ref{eq18}) that%
\begin{equation}
\frac{m(1+\lambda )}{2\lambda }a_{m+1}=(1-\beta )p_{m},  \label{eq19}
\end{equation}%
\begin{equation}
\frac{m(1+\lambda )}{2\lambda }\left( 2a_{2m+1}-a_{m+1}^{2}\right) +\frac{%
m^{2}(1-\lambda )}{4\lambda ^{2}}a_{m+1}^{2}=(1-\beta )p_{2m},  \label{eq20}
\end{equation}%
and%
\begin{equation}
-\frac{m(1+\lambda )}{2\lambda }a_{m+1}=(1-\beta )q_{m},  \label{eq21}
\end{equation}%
\begin{equation}
\frac{m(1+\lambda )}{2\lambda }\left[ (2m+1)a_{m+1}^{2}-2a_{2m+1}\right] +%
\frac{m^{2}(1-\lambda )}{4\lambda ^{2}}a_{m+1}^{2}=(1-\beta )q_{2m}.
\label{eq22}
\end{equation}%
From (\ref{eq19}) and (\ref{eq21}) we obtain%
\begin{equation}
p_{m}=-q_{m}.  \label{eq23}
\end{equation}%
and%
\begin{equation}
\frac{m^{2}(1+\lambda )^{2}}{2\lambda ^{2}}a_{m+1}^{2}=(1-\beta
)^{2}(p_{m}^{2}+q_{m}^{2}).  \label{eq24}
\end{equation}%
Adding (\ref{eq20}) and (\ref{eq22}), we have%
\begin{equation*}
\left[ \frac{m^{2}(1+\lambda )}{\lambda }+\frac{m^{2}\left( 1-\lambda
\right) }{2\lambda ^{2}}\right] a_{m+1}^{2}=(1-\beta )\left(
p_{2m}+q_{2m}\right) .
\end{equation*}%
Therefore, we obtain%
\begin{equation*}
a_{m+1}^{2}=\dfrac{2\lambda ^{2}(1-\beta )\left( p_{2m}+q_{2m}\right) }{%
m^{2}(2\lambda ^{2}+\lambda +1)}.
\end{equation*}%
Appying Lemma 1 for the coefficients $p_{2m}$ and $q_{2m}$, we obtain%
\begin{equation*}
\left\vert a_{m+1}\right\vert \leq \dfrac{2\lambda }{m}\sqrt{\dfrac{2\left(
1-\beta \right) }{2\lambda ^{2}+\lambda +1}}.
\end{equation*}%
Next, in order to find the bound on $\left\vert a_{2m+1}\right\vert ,$ by
subtracting (\ref{eq22}) from (\ref{eq20}), we obtain%
\begin{equation*}
\dfrac{2m(1+\lambda )}{\lambda }a_{2m+1}-\dfrac{m(m+1)(1+\lambda )}{\lambda }%
a_{m+1}^{2}=(1-\beta )\left( p_{2m}-q_{2m}\right) .
\end{equation*}%
Then, in view of (\ref{eq23}) and (\ref{eq24}) , appying Lemma 1 for the
coefficients $p_{2m},p_{m}$ and $q_{2m},q_{m},$ we have%
\begin{equation*}
\left\vert a_{2m+1}\right\vert \leq \dfrac{8(m+1)\lambda ^{2}\left( 1-\beta
\right) ^{2}}{m^{2}(1+\lambda )^{2}}+\frac{2\lambda \left( 1-\beta \right) }{%
m\left( 1+\lambda \right) }.
\end{equation*}%
This completes the proof of Theorem 5.
\end{proof}

If we set $\lambda =1$ in Theorems 3 and 5, then the classes $S_{\Sigma
_{m}}(\alpha ,\lambda )$ and $S_{\Sigma _{m}}(\beta ,\lambda )$ reduce to
the classes $S_{\Sigma _{m}}^{\alpha }$ and $S_{\Sigma _{m}}^{\beta }$ and
thus, we obtain following corollaries:

\begin{corollary}
( see \cite{Altinkaya and Yalcin 2015b}) Let $\ f$ given by (\ref{eq4}) be
in the class $S_{\Sigma _{m}}^{\alpha }\ \ \left( 0<\alpha \leq 1\right) $.
Then 
\begin{equation*}
\left\vert a_{m+1}\right\vert \leq \frac{2\alpha }{m\sqrt{\alpha +1}}
\end{equation*}%
and%
\begin{equation*}
\left\vert a_{2m+1}\right\vert \leq \frac{\alpha }{m}+\frac{2(m+1)\alpha ^{2}%
}{m^{2}}.
\end{equation*}
\end{corollary}

\begin{corollary}
( see \cite{Altinkaya and Yalcin 2015b}) Let $\ f$ given by (\ref{eq4}) be
in the class $S_{\Sigma _{m}}^{\beta }\ \left( 0\leq \beta <1\right) $. Then 
\begin{equation*}
\left\vert a_{m+1}\right\vert \leq \frac{\sqrt{2\left( 1-\beta \right) }}{m}
\end{equation*}%
and%
\begin{equation*}
\left\vert a_{2m+1}\right\vert \leq \frac{2(m+1)(1-\beta )^{2}}{m^{2}}+\frac{%
1-\beta }{m}.
\end{equation*}
\end{corollary}

The classes $S_{\Sigma _{m}}^{\alpha }$ and $S_{\Sigma _{m}}^{\beta }$ are
respectively defined as follows:

\begin{definition}
(see \cite{Altinkaya and Yalcin 2015b}) A function $f\in \Sigma _{m}$ given
by (\ref{eq4}) is said to be in the class $S_{\Sigma _{m}}^{\alpha }$ if the
following conditions are satisfied:%
\begin{equation*}
\begin{array}{cc}
f\in \Sigma ,\ \ \left\vert \arg \left( \dfrac{zf^{\prime }(z)}{f(z)}\right)
\right\vert <\dfrac{\alpha \pi }{2} & \left( 0<\alpha \leq 1,\ z\in U\right)%
\end{array}%
\ 
\end{equation*}%
and%
\begin{equation*}
\begin{array}{cc}
\left\vert \arg \left( \dfrac{wg^{\prime }(w)}{g(w)}\right) \right\vert <%
\dfrac{\alpha \pi }{2} & \left( 0<\alpha \leq 1,\ w\in U\right)%
\end{array}%
\end{equation*}%
where the function $g=f^{-1}.$
\end{definition}

\begin{definition}
( see \cite{Altinkaya and Yalcin 2015b}) A function $f\in \Sigma _{m}$ given
by (\ref{eq4}) is said to be in the class $S_{\Sigma _{m}}^{\beta }$ if the
following conditions are satisfied:%
\begin{equation*}
\begin{array}{cc}
f\in \Sigma ,\ \ \func{Re}\left( \dfrac{zf^{\prime }(z)}{f(z)}\right) >\beta
& \left( 0\leq \beta <1,\ \ z\in U\right)%
\end{array}%
\ 
\end{equation*}%
and%
\begin{equation*}
\begin{array}{cc}
\func{Re}\left( \dfrac{wg^{\prime }(w)}{g(w)}\right) >\beta & \left( 0\leq
\beta <1,\ \ w\in U\right)%
\end{array}%
\end{equation*}%
where the function $g=f^{-1}.$
\end{definition}

For \textit{one}-fold symmetric bi-univalent functions and $\lambda =1$,
Theorem 3 and Theorem 5 reduce to Corollary 10 and Corollary 11,
respectively, which were proven earlier by Murugunsundaramoorthy et al. \cite%
{Murugusundaramoorthy 2013}.

\begin{corollary}
Let $\ f$ given by (\ref{eq4}) be in the class $S_{\Sigma }^{\ast }(\alpha
)\ \ \left( 0<\alpha \leq 1\right) $. Then 
\begin{equation*}
\left\vert a_{2}\right\vert \leq \frac{2\alpha }{\sqrt{\alpha +1}}
\end{equation*}%
and%
\begin{equation*}
\left\vert a_{3}\right\vert \leq 4\alpha ^{2}+\alpha .
\end{equation*}
\end{corollary}

\begin{corollary}
Let $\ f$ given by (\ref{eq4}) be in the class $S_{\Sigma }^{\ast }(\beta )\
\left( 0\leq \beta <1\right) $. Then 
\begin{equation*}
\left\vert a_{2}\right\vert \leq \sqrt{2\left( 1-\beta \right) }
\end{equation*}%
and%
\begin{equation*}
\left\vert a_{3}\right\vert \leq 4(1-\beta )^{2}+(1-\beta ).
\end{equation*}
\end{corollary}

\end{document}